\documentclass[10pt]{amsart}
\usepackage{amssymb}
\usepackage{amsmath}
\usepackage{amsthm}
\usepackage{amsfonts}
\usepackage{mathrsfs}
\usepackage{titletoc}
\usepackage{amsxtra}
\usepackage{amsmath,amstext,amsfonts,verbatim, hyperref}

\newtheorem{theorem}{Theorem}

\newtheorem{lemma}{Lemma}
\newtheorem{remark}{Remark}
\newtheorem{proposition}{Proposition}
\newtheorem{example}{Example}
\newtheorem{thmx}{Theorem}

\title[]{On uniqueness of functions in the extended Selberg class with moving targets}
\author[J. Wang]{Jun Wang}
\address[Jun Wang]{School of Mathematics Sciences\\ Fudan University\\
Shanghai 200433\\ China}
\email{majwang@fudan.edu.cn}
\author[Q.Y.Wang]{Qiongyan Wang}
\address[Qiongyan Wang]{School of Mathematical Sciences\\ Peking University\\ Beijing, 100871\\ China}
\email{qiongyanwang@pku.edu.cn}
\author[X. Yao]{Xiao Yao}
\address[Xiao Yao]{School of Mathematical Sciences and LPMC\\ Nankai University\\ Tianjin, 300071\\ China}
\email{yaoxiao@nankai.edu.cn}

\thanks{The first author was supported by National Natural Science Foundation of China (Grant No.11771090). The third author was supported  by National Key R\&D Program of China (2020YFA0713300) and NSFC of China(No.11901311).}

\begin{document}
\maketitle

\begin{abstract}
We study the question of when two functions $L_1,L_2$ in the extended Selberg class are identical in terms of the zeros of $L_i-h(i=1,2)$.
Here, the meromorphic function $h$ is called moving target. With the assumption on the growth order of $h$, we prove that $L_1\equiv L_2$ if $L_1-h$
and $L_2-h$ have the same zeros counting multiplicities. Moreover, we also construct some examples to show that the assumption is necessary. Compared with the known methods in the literature of this area, we developed a new strategy which is based on the transcendental directions first proposed in the study of distribution of Julia set in complex dynamical system. This may be of independent interest.
\end{abstract}

\section{Introduction and main results}

This paper concerns the uniqueness problem related to moving
targets for $L$-functions, which are Dirichlet series with the Riemann zeta function as the prototype. The Selberg class $\mathcal{S}$ of $L$-functions consists of all Dirichlet series $L(s)=\sum_{n=1}^{\infty}{a(n)\over n^s}$ of a complex variable $s$ with $a(1)=1$,
satisfying the following axioms (see \cite{Se,St}):\medskip

(i) Ramanujan hypothesis.  $a(n)\ll n^{\varepsilon}$ for any $\varepsilon>0$.\vskip 1mm

(ii) Analytic continuation.  There exists a non-negative integer $k$ such that $(s-1)^{k}\emph{L}(s)$ is an entire function of finite order.\vskip 1mm

(iii) Functional equation.   $\emph{L}(s)$ satisfies a functional equation of type
$$\Lambda_\emph{L}(s)=\omega\overline{\Lambda_\emph{L}(1-\overline{s})},$$\\
where
$\Lambda_\emph{L}(s)=\emph{L}(s)Q^{s}\prod\limits_{j=1}^{K}\Gamma(\lambda_{j}s+\mu_{j})$
with the constants $Q,\lambda_j,\mu_j,\omega$ having $Q>0,\lambda_{j}>0, \rm{Re}\,\mu_{j}\geq0$ and $|\omega|=1$.\vskip 1mm

(iv) Euler product.  $\emph{L}(s)$ has a product representation $\emph{L}(s)=\prod\limits_{p}\emph{L}_{p}(s)$,\\
where the product is over all primes $p$, and
$$\emph{L}_{p}(s)= \exp\left(\sum\limits_{j=1}^{\infty}\frac{b(p^{j})}{p^{js}}\right)$$
with coefficients $b(p^{j})$ satisfying $b(p^{j})\ll p^{j\theta}$ for some $\theta<\frac{1}{2}$.\vskip 2mm
\par The extended Selberg class, denoted by $\mathcal{S}^{\sharp}$, is the set of all Dirichlet series $\emph{L}(s)$ satisfying only the axioms (i)-(iii) (see \cite{St}).
The degree $d_L$ of $L(s)$ is defined to be $$d_L=2\sum\limits_{j=1}^{K}\lambda_{j},$$ where $K$ and $\lambda_{j}$ are the numbers in the axiom (iii). The poles of $\Gamma(\lambda_{j}s+\mu_{j})$ occur at points $s=-\frac{m+\mu_j}{\lambda_j}(m=0,1,2,\cdots)$, which are also zeros of $L(s)$ with the possible exception of $s=0$. All these zeros have non-positive real part, and may have multiplicity greater than one. They are so-called ``trivial" zeros of $L(s)$, see \cite{GHK}. Throughout this paper, what we call an $\emph{L}$-function is a Dirichlet series which belongs to the extended Selberg class $\mathcal{S}^{\sharp}$. In particular, we do not assume the Euler product axiom. The results obtained in this paper particularly apply to the Selberg class $\mathcal{S}$.\vskip 2mm

As an extension of the Riemann zeta function, $\emph{L}$-functions are the important objects in number theory, algebraic geometry etc. Value distribution and uniqueness of $\emph{L}$-functions have been studied extensively
(see e.g. \cite{CY,GHK,Ki,KL,Kn,Li1,Li2,St}). Generally, Nevanlinna  \cite{Hayman} proved that any non-constant meromorphic function can be uniquely determined by five values, which is called Five-value Theorem. It means that meormorphic functions $f$ and $g$ are identical when $f,g$ take the same five values at the same points. While for $\emph{L}$-functions, the number of the values can be reduced to just one when counting multiplicities (see \cite{St,HL}), which is stated in the following theorem.\vskip 2mm

\begin{thmx}\label{St-HL}
Two L-functions $L_1$ and $L_2$ are equal if
 $L_1-a$ and $L_2-a$ have the same zeros counting multiplicities for a complex number $a(\not=1)$.
\end{thmx}

Steuding's original statement has no restriction on $a$, see \cite[Theorem 7.11]{St}. Lately, an example given by Hu and Li \cite{HL}, that is $\emph{L}_1(s)=1+\frac{2}{4^s}$ and $\emph{L}_2(s)=1+\frac{3}{9^s}$, shows that $a\not=1$ cannot be dropped in Theorem \ref{St-HL}. When $a$ is replaced by a rational function $R(s)$, Cardwell and Ye \cite{CY} proved the uniqueness conclusion under the assumption that
$$\lim_{s\rightarrow \infty}R(s)\neq 1.$$
Hu and Li strengthened their result by improving $R(s)$ to a meromorphic function $h(s)$, see \cite[Theorem 3]{HL}.
\begin{thmx}\label{HL}
Let $h$ be a meromorphic function of finite order such that
\begin{equation}\label{eq-condition}
\lim_{\rm{Re}\,s\rightarrow +\infty}h(s)=a\not=1
\end{equation}
($a$ may be $\infty$). Two L-functions $L_1$ and $L_2$ are equal if $L_1-h$ and $L_2-h$ have the same zeros counting multiplicities.
\end{thmx}

Here, the term `order' comes from value distribution theory. The standard notation and well-known theorems in the theory can be found in \cite{Gol,Hayman,YY}. Let $f$ be a meromorphic function. Define the proximity function $m(r,f)$ and the integrated counting function of poles $N(r,f)$ by
$$m(r,f):=\frac{1}{2\pi}\int_{0}^{2\pi}\log^+|f(re^{i\theta})|d\theta,$$
$$N(r,f):=\int_0^r\frac{n(t,f)-n(0,f)}{t}dt+n(0,f)\log r,$$
where $\log^+ x=\max\{\log x,0\}$ for $x>0$, and $n(t,f)$ denotes the number of poles in $|z|\leq t$ counting multiplicity.
The sum $T(r, f):=m(r,f)+N(r, f)$ is the Nevanlinna characteristic, then the order $\rho(f)$ and the lower order $\mu(f)$ are defined by
$$\rho(f):=\limsup\limits_{r\to\infty}\frac{\log^{+} T(r,f)}{\log r},\ \ \ \mu(f):=\liminf\limits_{r\to\infty}\frac{\log^{+} T(r,f)}{\log r}.$$

When $h$ is non-constant, the condition (1) in Theorem \ref{HL} is necessary to be kept due to the example in which $h$ can be an arbitrary $L$-function. Moreover, Hu and Li also constructed an example to show that Theorem \ref{HL} need not hold if $h$ is of infinite order, see \cite{HL}. Here, $h$ can be seen as a moving target. Usually, a moving target of a function is assumed to grow more slowly than the function, then it is called a small function. It is interesting that there is no such growth restriction for $h$ in \cite{HL}, and $h$ may grow more quickly than $L_1$ and $L_2$. \vskip 2mm

Moving targets have appeared in Nevanlinna's version of the second main theorem (SMT) with three small functions. Chuang \cite{Chuang} obtained the general version of SMT with $q(\geq 3)$ small functions for entire functions in 1964.  The meromorphic case was solved by Osgood and Steinmetz independently\cite{Osgood,Stei} in the 1980's. It has been open for quite a long time to establish the small function version of SMT truncated to the level 1.\vskip 2mm

Without this tool, Nevanlinna's strategy in proving Five-value Theorem can not work out the small function version of Five-value Theorem, which was solved by Li and Qiao \cite{LiY} with a different approach. Finally, the SMT with small functions truncated to the level $1$ was obtained in a remarkable paper of Yamanoi \cite{Yamanoi}.  According to the Vojta's dictionary, it seems that the moving target problem has some close and subtle relation with number theory. This is our initial motivation to study the uniqueness problem of $L$-functions with moving targets.\vskip 2mm

The condition (1) and `finite order' on $h$ seem to be two different ways to describe $h$, and both of them can not be deleted completely in Theorem \ref{HL}. The question naturally arises as to whether these two requirements could be weakened in some sense. In Hu and Li's example for the condition (1), $h$ is an arbitrary $L$-function, and it is known that $\emph{L}$-functions are of order one. This motivates us to pay attention only to the order of the moving target $h$. \vskip 2mm

 In this paper, we base our uniqueness results of $\emph{L}$-functions on the above Theorem \ref{St-HL}, Theorem \ref{HL} and their remarks, and try to find a new single condition on $h$ in terms of its order. Since $\emph{L}$-functions behave almost like entire functions except one possible pole at $s=1$, we firstly consider the non-constant entire $h$, and have the following theorem.
\begin{theorem}\label{thm-entire}
Let $h$ be a non-constant entire function with $\rho(h)\neq 1$. Two L-functions $L_1$ and $L_2$ are equal if $L_1-h$ and $L_2-h$ have the same zeros counting multiplicities.
\end{theorem}

Now in Theorem \ref{thm-entire}, $h$ can be of infinite order. The condition $\rho(h)\neq 1$ here is best possible in the sense by taking $$\emph{L}_1(s)=1+\frac{2}{4^s},\,\,\emph{L}_2(s)=1+\frac{3}{9^s},\,\,h(s)=1+\frac{2}{4^s}+\frac{3}{9^s}.$$ It is clear that
$\emph{L}_1-h$ and $\emph{L}_2-h$ do not have any zeros, and $\rho(h)=1$.\vskip 2mm

\par The proof of Theorem 1 essentially used transcendental directions, which was introduced by two of the authors when investigating
the distribution of Julia set in the transcendental iteration theory \cite{Wang-Yao}.  To the best of our knowledge, it seems the first time that the transcendental
direction has been applied in the uniqueness study of $L$-functions.\vskip 2mm

When $h$ is rational, we obtain a slight generalization of Cardwell and Ye's result \cite{CY} by removing the assumption on the limit of the rational moving target $h$ as $s\to\infty$.
\begin{theorem}\label{thm-rational}
Let $h$ be a rational function such that $h\not\equiv 1$. Two L-functions $L_1$ and $L_2$ are equal if $L_1-h$ and $L_2-h$ have the same zeros counting multiplicities.
\end{theorem}

The next part is for transcendental meromorphic $h$ with at least one pole. We obtain two uniqueness theorems for $\emph{L}$-functions satisfying the same functional equation or with positive degree, if the order of $h$ takes finite value outside $\mathbb{N}$. Here, $\mathbb{N}$ denotes the set of all positive integers.
\begin{theorem}\label{thm-meromorphic}
Let $h$ be a transcendental meromorphic function with at least one pole and $\rho(h)\notin \mathbb{N}\cup\{\infty\}$.
For two L-functions $L_1$ and $L_2$ satisfying the same functional equation, if $L_1-h$ and $L_2-h$ have the same zeros counting multiplicities, then $L_1\equiv L_2$.
\end{theorem}

For $L$-functions not possessing the same functional equation, once they have positive degree, the following theorem holds correspondingly.
\begin{theorem}\label{thm-meromorphic-positive-degree}
Let $h$ be a transcendental meromorphic function with at least one pole and $\rho(h)\notin \mathbb{N}\cup\{\infty\}$.
For two $L$-functions $L_1$ and $L_2$ with positive degree, if $L_1-h$ and $L_2-h$ have the same zeros counting multiplicities,
then $L_1\equiv L_2$.
\end{theorem}

The condition $\rho(h)\notin \mathbb{N}\cup\{\infty\}$ in above two theorems can not be dropped, which is shown by the following two examples.
\begin{example}
For any two distinct $L$-functions $L_1,L_2$ and a positive integer $m$, we set
$$h_m(s)=\frac{\exp(s^m)L_2(s)-L_1(s)}{\exp(s^m)-1},\,\,\text{then}\quad \frac{L_1-h_m}{L_2-h_m}=\exp(s^m).$$
It is clear that $L_1-h_m$ and $L_2-h_m$ have the same zeros counting multiplicities. When $m\geq 2$, we know
$\rho(h_m)=m$ since $\rho(L_1)=\rho(L_2)=1$. Similarly, set $$h_{\infty}(s)=\frac{\exp(e^s)
L_2(s)-L_1(s)}{\exp(e^s)-1},$$
then $L_1-h_\infty$ and $L_2-h_\infty$ also have the same zeros counting multiplicities, and $\rho(h_{\infty})=\infty$.
\end{example}
\begin{example}
It can be seen in \cite{KL} that the two distinct $L$-functions
\begin{equation}
L_1(s)=1+\frac{2}{4^s},\quad L_2(s)=1+\frac{1}{2^s}+\frac{2}{4^s}
\end{equation}
satisfy the same functional equation
$$
2^{s}L(s)=2^{1-s}\overline{L(1-\overline{s})}.
$$
Clearly, $L_1-h_1$ and $L_2-h_1$ have the same zeros counting multiplicities, where $h_1$ is defined in Example 1, and $d_{L_1}=d_{L_2}=0$.
Rewrite $h_1$ as
$$h_1(s)=1+\frac{2}{4^s}+\frac{2^{-s}}{e^s-1}.$$
It is easy to see $N(r,h_1)=N(r,\frac{1}{e^s-1})+\log r$, so $\rho(h_1)=1$.
\end{example}

\section{Preliminaries}
For convenience of readers, we begin the section with two fundamental theorems in value distribution theory, see \cite{Gol,Hayman,YY}.
The Second Main Theorem stated here is only a simple version since it is already enough for us. \vskip 2mm
\par\noindent {\bf First Main Theorem.}\, Let $f$ be a meromorphic function and let $c$ be any complex number. Then
$$T(r,\frac{1}{f-c})=T(r,f)+O(1),$$
where the error term depends on $f$ and $c$.\vskip 2mm
\par\noindent {\bf Second Main Theorem.}\, Let $f$ be a meromorphic function. For any distinct $a_1,a_2,\cdots,a_q\in \mathbb{C}\cup\{\infty\}$ and an integer $q\geq 3$, we have
$$(q-2)T(r,f)\leq \sum_{j=1}^{q}N(r,\frac{1}{f-a_j})+S(r,f),$$
where $S(r,f)$ is the small error term,
$$S(r,f)=O(\log rT(r,f)),\quad r\to\infty,\,\,r\not\in E$$
and the set $E$ has finite length.\vskip 2mm
\par Given a meromorphic function $f$, for any $a\in \mathbb{C}\cup\{\infty\}$, a value $\theta\in [0, 2\pi)$ is said to be a $a$-value limiting direction of $f$ if there exists an unbounded sequence of $\{s_n\}_{n=1}^{\infty}$ such that
$$
\lim_{n\to\infty}\arg s_n=\theta\quad \text{and}\quad \lim_{n\rightarrow\infty}f(s_n)=a.
$$
When $a=\infty$, if further
$$
\lim_{n\rightarrow\infty}\frac{\log |f(s_n)|}{\log |s_n|}=+\infty,
$$
we call such $\theta$ a transcendental direction of $f$.
\begin{example}
For $f(s)=(s-1)^me^{as+b}$ with some integer $m$ and constants $a\not=0,b\in\mathbb{C}$, it is clear that $\theta$ with positive ${\rm Re}(ae^{i\theta})$ is a transcendental direction of $f$ while $\theta$ with negative ${\rm Re}(ae^{i\theta})$ is a $0$-value limiting direction of $f$.
\end{example}
The set of all transcendental directions is denoted by
$TD(f)$, and the following lemma gave the lower bound to the Lebesgue measure of $TD(f)$.

\begin{lemma}\cite[Theorem 5]{Wang-Yao}\label{wang-yao}  Let $f$ be a transcendental meromorphic function. If $\mu(f)<\infty$ and
 $$0<\delta(\infty,f):=1-\limsup\limits_{r\to\infty}\frac{N(r,f)}{T(r,f)},$$  then
\[meas(TD(f))
\ge\min\Big\{2\pi,\frac{4}{\mu(f)}\arcsin
\sqrt{\frac{\delta(\infty,f)}{2}}\Big\}.\]
\end{lemma}

Transcendental direction was introduced in \cite{Wang-Yao} to study the radial distribution of Julia set for meromorphic functions. We will see that this notion is also useful in discussing the uniqueness of $L$-functions. For example, by Lemma \ref{wang-yao}, we can estimate the positive constant $Q$ in the functional equation satisfied by $L$-functions of degree zero.

\begin{proposition}\label{prop-positive}
Let $L(s)$ be a non-constant $L$-function of degree zero such that
$$
Q^s L(s)=\omega Q^{1-s}\overline{L(1-\overline{s})},
$$
where $|\omega|=1$ and $Q> 0$. Then we have $Q>1$.
\end{proposition}
\begin{proof}
 We first assume that $Q\leq 1$. Noting that
 $$|L(s)|\leq |\omega| Q^{1-2{\rm Re} s}\sum_{n=1}^{\infty} \frac{|a(n)|}{n^{1-{\rm Re} s}},$$
we know that there exists $K_0>0$  such that
$|L(s)|\leq K_0$ for any $s$ with ${\rm Re} s<-2$ by Ramanujan hypothesis. This together with the fact that
$$
L(s)\to 1,\quad \text{as}\quad {\rm Re}\,s\to+\infty
$$
yields that $TD(L)\subseteq(\frac{\pi}{2}, \frac{3\pi}{2})$.
Since $\rho(L)=1$, we have $meas(TD(L))\geq \pi$ from Lemma \ref{wang-yao}, which leads to a contradiction. Therefore, we get $Q>1$.
\end{proof}

\begin{lemma}\label{lemma-estimate}
Let $L(s)$ be a non-constant $L$-function of positive degree $d_L$ in extended Selberg class, and let its functional equation be
$$\Lambda_L(s)=\omega\overline{\Lambda_L(1-{\bar
s})},\quad\text{where}\quad \Lambda_L(s)=L(s)Q^s\prod_{j=1}^K\Gamma(\lambda_j
s+\mu_j)$$ with $Q>0, \lambda_j>0, Re\,\mu_j\geq 0$ and
$|\omega|=1$. Then for any $0<\delta<\pi/4$, we have
\begin{equation*}\label{estimate in left half plane}
\begin{split}
\log |L(re^{i\theta})|=&-\big(d_L\cos \theta\big)r\log r-\big[2\log Q +2\sum_{j}\lambda_j\log\lambda_j-d_L\big]r\cos\theta\\
&+(\theta-\frac{\pi}{2})\big(d_L\sin\theta\big)r+O(\log r)
\end{split}
\end{equation*}
as $r\to\infty$, uniformly for $\theta\in (\frac{\pi}{2}+\delta, \pi-\delta)\cup(\pi+\delta, \frac{3}{2}\pi-\delta)$.
\end{lemma}
\begin{proof}
It follows from the functional equation, we have
\begin{equation}\label{10-26-1}
|L(s)|=Q^{1-2 {\rm Re}s}|L(1-\overline{s})|
\prod_{j=1}^{K}\frac{|\Gamma(\lambda_j(1-\overline{s})+\mu_j)|}
{|\Gamma(\lambda_j s+\mu_j)|}
\end{equation}
By Stirling's formula (see, e.g. \cite[p.151]{13}),
\begin{equation}\label{10-26-2}
\log \Gamma(z)=(z-\frac{1}{2})\log z-z+\frac{1}{2}\log (2\pi)+O\left(\frac{1}{|z|}\right)
\end{equation}
as $|z|\to\infty$, uniformly for $z\in \Omega_\delta=\{z:|\arg z|\leq \pi-\delta\}$ with any $0<\delta<\pi/2$. Since $L$ satisfies Ramanujan hypothesis and $a(1)=1$, there exists positive constants $\sigma_0,C_0$ such that
\begin{equation}\label{10-26-3}|L(1-\overline{s})-1|<C_02^{{\rm Re}s}\leq C_0 |s|^{-1} \end{equation}
for ${\rm Re}s<-\sigma_0$ and $\arg s\in(\pi/2+\delta,3\pi/2-\delta)$. Taking \eqref{10-26-2} and \eqref{10-26-3} into \eqref{10-26-1}
 we have
\begin{equation}\label{10-26-4}
\begin{split}
\log |L(s)|=&\sum_{j=1}^{K}{\rm Re}\Big\{\log \Gamma(\lambda_j(1-\overline{s})+\mu_j)-
\log \Gamma(\lambda_j s+\mu_j)\Big\}\\
&+(1-2{\rm Re}s)\log Q+O(|s|^{-1})\\
=&\sum_{j=1}^{K}{\rm Re}\Big\{-\lambda_j \overline{s}\log(-\overline{s})-2\lambda_js\log\lambda_j
-\lambda_js\log s+2\lambda_j s\Big\}\\
&-2{\rm Re}s\log Q+O(\log |s|)
\end{split}
\end{equation}
as $s\to\infty$, if all $\lambda_j(1-\overline{s})+\mu_j,\lambda_j s+\mu_j\in \Omega_\delta$ and $\arg s\in(\pi/2+\delta,3\pi/2-\delta)$.
We write $s=re^{i\theta}$, it is clear that $$\arg(\lambda_j(1-\overline{s})+\mu_j)\sim \arg(-\overline{s})\quad \text{and} \quad \arg(\lambda_j s+\mu_j)\sim\arg s$$ as $s\to\infty$. Then take $\theta\in (\frac{\pi}{2}+\delta, \pi-\delta)\cup(\pi+\delta, \frac{3}{2}\pi-\delta)$, we get the desired estimate from \eqref{10-26-4} where
$-\cos\theta>0$.
\end{proof}

\begin{proposition}\label{prop-degree}
Let $h$ be a meromorphic function such that $\rho(h)\in [0, 1)$, and let $L_1,L_2$ be two $L$-functions. If $L_1-h$ and $L_2-h$ have the same zeros counting multiplicities, then $L_1$ and $L_2$ have the same degree.
\end{proposition}
\begin{proof}
By Hadamard factorization theorem, we have
\begin{equation}\label{10-26-5}\frac{L_1(s)-h(s)}{L_2(s)-h(s)}=(s-1)^{m}\exp(as+b)\end{equation}
for some integer $m$ and constants $a, b\in\mathbb{C}$. It is clear that $$L_1=h+(L_2-h)(s-1)^{m}e^{as+b}.$$
Recall the basic fact, see e.g. \cite[p.5]{Hayman}, that
\begin{equation}\label{23-9-19}T(r,fg)\leq T(r,f)+T(r,g),\,\, T(r,f+g)\leq T(r,f)+T(r,g)+\log 2.\end{equation}
By this fact and $\rho(h)<1$, it follows from \eqref{10-26-5} that
$$
T(r, L_1)\leq T(r, L_2)+2T(r, h)+O(r)+O(\log r)=T(r,L_2)+O(r).
$$
Similarly, we also have
$$
T(r, L_2)\leq T(r, L_1)+O(r).
$$
Noting \cite[Theorem 7.9]{St}, that is $$T(r, L)=\frac{d_L}{\pi}r\log r +O(r),$$
the above inequalities between $T(r,L_1)$ and $T(r,L_2)$ leads $d_{L_1}=d_{L_2}$.
\end{proof}

For a canonical product, Borel has given its lower bound outside neighborhoods of its zeros (see \cite[Theorem 5.13]{Hayman}). Based on the fact, it is known that there is a upper bound for the modulus of a meromorphic function when $z$ is away from its poles. For convenience of readers, we state the result in the following lemma, which
comes from \cite[Lemma 1]{Ch} and its proof.

\begin{lemma}\label{meromorphic-estimate}
Suppose that $f(z)$ is a meromorphic function with $\rho(f)=\beta<\infty$ and the sequence of non-zero poles $\{p_j\}_{j=1}^\infty$ arranged according to increasing moduli. Then for any given $\varepsilon>0$, set $O=\bigcup_{j=1}^\infty\{z:\,|z-p_j|\leq |p_j|^{-(\beta+\frac{\varepsilon}{2})}\}$, we have
$$|f(z)|\leq \exp\{r^{\beta+\varepsilon}\}\qquad \text{for}\quad z\not\in O,$$
as $|z|=r\to\infty$. Specially, the set $E=\{|z|: z\in O\}$ is of finite Lebesgue measure.
\end{lemma}

\section{Proof of Theorems \ref{thm-entire} and \ref{thm-rational}}
We will first prove Theorem \ref{thm-rational}. Then in the proof of Theorem \ref{thm-entire}, we can only focus of
treating the case $h$ transcendental. \vskip 2mm
\noindent{\bf Proof of Theorem \ref{thm-rational}.}\,\,Due to the Theorem of Cardwell and Ye \cite{CY}, we may assume the condition
$$\lim\limits_{s\rightarrow\infty}h(s)=1.$$
Thus, $h(s)$ has the Laurent's expansion at $s=\infty$ as
$$
h(s)=1+\sum_{k=j_0}^{\infty}\frac{a_k}{s^{k}},
$$
where $j_0$ is some positive integer, $a_{j_0}\neq 0$. It is obvious that there exist positive constants $\sigma_0,C_0$ such that for $s\in \Omega=\{z:{\rm Re} z>\sigma_0,|{\rm Im} z|\leq {\rm Re} z\}$, we have
$$|L_j(s)-1|\leq C_0 2^{-{\rm Re}s},\quad (j=1,2)$$
$$2^{-(1+j_0)}\frac{|a_{j_0}|}{({\rm Re}s)^{j_0}}\leq|h(s)-1|\leq 2\frac{|a_{j_0}|}{({\rm Re}s)^{j_0}}.$$
By the fact that $$\lim_{x\to+\infty}\frac{2^{-x}}{x^{-j_0}}=\lim_{x\to+\infty}\frac{x^{j_0}}{2^x}=0,$$
then we get
$$\lim\limits_{s\in\Omega\to\infty}\frac{L_1-h}{L_2-h}=\lim\limits_{s\in\Omega\to \infty}\frac{(L_1-1)-(h-1)}{(L_2-1)-(h-1)}=1.$$
Noting that \eqref{10-26-5} and the above limit, it yields immediately that $1$ is the asymptotic value of $\varphi(s)=(s-1)^{m}e^{as+b}$. When $a\not=0$, $\varphi$ has only two asymptotic values
$0,\infty$, so we have $a=0$. Clearly, $(s-1)^m e^b\to 0,\infty,e^b$ as $s\to \infty$ according to $m<0,m>0$ and $m=0$, respectively. From this, we further get $m=b=0$. Therefore, $L_1\equiv L_2$, which completes the proof.\vskip 2mm

\noindent {\bf Proof of Theorem \ref{thm-entire}.} \,\,We now assume that $L_1\not\equiv L_2$ in the following, and complete the proof by reduction to absurdity.\vskip 2mm
\par We first deal with the case that $\rho(h)>1$. Since $L_1-h$ and $L_2-h$ have the same zeros counting multiplicities, we know that
\begin{equation}\label{10-28-3}
\varphi(s):=\frac{L_1-h}{L_2-h}=(s-1)^{m}e^{\phi},
\end{equation}
where $\phi(s)$ is an entire functions, and $m$ is some integer.
It follows from \eqref{23-9-19} and Nevanlinna's first main theorem that
\begin{equation}\label{eq-1}
\begin{aligned}
T(r, h)&\leq T\big(r, \frac{1}{L_2-h}\big)+T(r, L_2)+O(1)\\
&\leq T\big(r, \frac{L_1-L_2}{L_2-h}\big)+T(r, L_1-L_2)+T(r, L_2)+O(1)\\
&\leq T(r, \varphi-1)+2T(r, L_2)+T(r, L_1)+O(1)\\
&\leq T(r, \varphi)+2T(r, L_2)+T(r, L_1)+O(1).
\end{aligned}
\end{equation}
Next, we will estimate $T(r,\varphi)$ by the second main theorem. Since
$$N(r,fg)\leq N(r,f)+N(r,g),\,\,N(r,f+g)\leq N(r,f)+N(r,g),$$
see e.g.\cite[p.5]{Hayman}, obviously we have
\begin{equation*}\label{eq-2}
\begin{aligned}
N(r, \frac{1}{\varphi-1})=N\big(r, \frac{L_2-h}{L_1-L_2}\big)\leq N(r,L_2)+N(r, \frac{1}{L_2-L_1})+O(1).
\end{aligned}
\end{equation*}
Moreover, noting that $N(r, 1/\varphi)+N(r, \varphi)=|m|\log r$, we obtain
\begin{equation}\label{10-28-1}
\begin{split}
T(r,\varphi)\leq &N(r,\frac{1}{\varphi})+N(r,\varphi)+N(r,\frac{1}{\varphi-1})+o(T(r,\varphi))\\
\leq & N(r,L_2)+N(r,\frac{1}{L_2-L_1})+O(\log r)+o(T(r,\varphi))\\
\leq &T(r,L_1)+2T(r,L_2)+O(\log r)+o(T(r,\varphi)),
\end{split}
\end{equation}
possibly outside a set of $r$ with finite Lebesgue measure. Substituting \eqref{10-28-1} into \eqref{eq-1}, together with \cite[Theorem 7.9]{St}, it yields
$$
T(r, h)=O(T(r,L_1)+T(r,L_2)+\log r)=O(r\log r),
$$
which implies that $\rho(h)\leq 1$, a contradiction.\vskip 2mm
\par Next, we turn to the case that $\rho(h)<1$. By Hadamard factorization theorem, $\phi(s)$ is a polynomial of degree at most one, $as+b$ say, so $\varphi(s)=(s-1)^m\exp(as+b)$ with  constants $a,b\in \mathbb{C}$. At the same time, since $\mu(h)\leq \rho(h)< 1$, $N(r,h)=0$, applying Lemma 1 to $h$ gives $meas(TD(h))>\pi$. Let $$\mathcal{B}=TD(h)\backslash \big[\frac{\pi}{2}, \frac{3\pi}{2}\big],$$ then $meas(\mathcal{B})>0$. Due to the fact $
\lim_{{\rm Re}s\rightarrow +\infty}L(s)=1$, clearly for any $\theta
\in \mathcal{B}$, there exist a sequence $\{s_n\}_{n=1}^{\infty}$
such that $\lim_{n\to\infty}\arg s_n=\theta$ and $$\lim_{n\rightarrow\infty}\frac{\log|h(s_n)|}{\log |s_n|}=+\infty,\quad \text{so}\,\,\lim_{n\rightarrow\infty}\varphi(s_n)=1.$$
This means that every $\theta\in \mathcal{B}$ is a $1$-value limiting direction of $\varphi$. When $a\not=0$,
clearly $\varphi(s)$ has at most two possible $1$-value limiting directions $\theta$, which satisfy ${\rm Re}(ae^{i\theta})=0$. Since $meas(\mathcal{B})>0$ this is impossible, so we have $a=0$. In that case, we have $m=0$ and $\varphi(s)=e^b=1$ which contradicts the assumption $L_1\not\equiv L_2$.
\vskip 2mm
\par Therefore, from the above argument, we have $L_1\equiv L_2$ if $\rho(h)\not=1$.\vskip 2mm
\begin{remark}\label{remark-pole}
The method in the above proof also works for meromorphic $h$ which satisfies $N(r,h)=O(r\log r)$ and if $\mu(h)<1$,
\begin{equation}\label{10-28-4}
\arcsin\sqrt{\frac{\delta(\infty, h)}{2}}>\frac{\mu(h)\pi}{4},\quad \text{that is}\quad \frac{4}{\mu(h)}\arcsin\sqrt{\frac{\delta(\infty,h)}{2}}>\pi.
\end{equation}
\par When $\rho(h)>1$, the proof is similar but has some minor changes. For example, \eqref{10-28-3} still holds since a pole of $h$, excepted for $s=1$, cannot be a pole of $\varphi$. At the same time, the poles of $1/(\varphi-1)$ possibly contain the poles of $h$ except the zeros of $L_2-L_1$ and the point $s=1$. Thus, \eqref{10-28-1} becomes $$T(r,\varphi)\leq N(r,h)+T(r,L_1)+2T(r,L_2)+o(T(r,\varphi))=O(r\log r).$$
It follows from \eqref{eq-1} that $T(r,h)=O(r\log r)$, so $\rho(h)\leq 1$, a contradiction.
\par When $\rho(h)\in [0, 1)$, the condition \eqref{10-28-4} also guarantees that $meas(\mathcal{B})>0$ by Lemma \ref{wang-yao}, so the proof just follows directly.
\end{remark}
\section{Proof of Theorems \ref{thm-meromorphic} and \ref{thm-meromorphic-positive-degree}}

Since in the proof of these two theorems, the discussion is the same for the case $\rho(h)>1$, we will first prove Theorem \ref{thm-meromorphic}. Then we focus on dealing with the case that $\rho(h)<1$ in the proof of Theorem \ref{thm-meromorphic-positive-degree}.\vskip 2mm

\noindent {\bf Proof of Theorem \ref{thm-meromorphic}.}\,\,Similarly as \eqref{10-28-3}, it follows from $\rho(h)<\infty$ and Remark \ref{remark-pole} that there exists a polynomial $p$ of degree $d$ and some integer $m$ such that
$$
\varphi(s)=\frac{L_1-h}{L_2-h}=(s-1)^m e^p.
$$
We assume that $L_1\not\equiv L_2$, then $\varphi(s)\not\equiv 1$. From this, $h$ can be solved out as
$$h=\frac{\varphi L_2-L_1}{\varphi-1}=\frac{(s-1)^m e^p L_2-L_1}{(s-1)^m e^p-1}.$$
It is clear that $\rho(h)\leq \max\{d, 1\}$. When $\rho(h)>1$, obviously $1<\max\{d,1\}$, so $d\geq 2$. Thus $\rho(h)=d\in \mathbb{N}$, which is a contradiction.\vskip 2mm
\par Now, we turn to the case that $\rho(h)<1$, then again $
\varphi(s)=(s-1)^{m}e^{as+b}$ for some $a, b\in\mathbb{C}$, $m\in\mathbb{Z}$.
Since $L_1$ and $L_2$ have the same functional equations, we have $d_{L_1}=d_{L_2}$ and
 \begin{equation}\label{same-funct-eq}
\frac{L_1(s)}{L_2(s)}=\frac{\overline{L_1(1-\overline{s})}}
{\overline{L_2(1-\overline{s})}}.
\end{equation} At the same time, applying Lemma \ref{meromorphic-estimate} to $h$ yields that for any given $\varepsilon>0$,
\begin{equation}\label{eq-upper}
|h(s)|\leq \exp\{r^{\rho(h)+\varepsilon}\}
\end{equation}
holds for any $z$ satisfying $|s|=r$ outside a set $E$ of finite Lebesgue measure and $r\to\infty$.
If $d_{L_1}=d_{L_2}>0$, then by Lemma \ref{lemma-estimate},  we have
\begin{equation}\label{positive-degree}
\lim_{r\rightarrow\infty}\frac{\log |L_j(r e^{i\theta})|}{(-d_{L_j}\cos\theta) r\log r}=1,\quad (j=1,2)
\end{equation}
for any $\theta\in \mathcal{I}_{\delta}:=(\frac{\pi}{2}+\delta,\pi-\delta)\cup (\pi+\delta,\frac{3}{2}\pi-\delta)$. Combining \eqref{eq-upper} and \eqref{positive-degree} gives that there exists a sequence of $\{s_n=r_ne^{i\theta}\}_{n=1}^{\infty}$ with $\theta\in \mathcal{I}_{\delta}$ and $r_n\not\in E$ such that
\begin{equation}\label{eq-key-1}
\lim_{n\rightarrow\infty}
\frac{L_j(r_n e^{i\theta})-h(r_ne^{i\theta})}{L_j(r_n e^{i\theta})}=1,\quad (j=1,2).
\end{equation}
Noting that $\lim\limits_{{\rm Re}s\to+\infty}L_j(s)=1(j=1,2)$, together with \eqref{same-funct-eq} and \eqref{eq-key-1}, we have
\begin{equation}\label{eq-key-2}
\lim_{n\rightarrow\infty}\varphi(s_n)=\lim_{n\rightarrow\infty}\frac{L_1(r_n e^{i\theta})}
{L_2(r_n e^{i\theta})}=\lim_{n\rightarrow\infty}\overline{\left(\frac{L_1(1-r_n e^{-i\theta})}
{L_2(1-r_n e^{-i\theta})}\right)}=1,
\end{equation}
which means that every $\theta\in \mathcal{I}_{\delta}$ is $1$-value limiting direction of $\varphi$. From this, we can also get a contradiction by the
argument in the proof of Theorem \ref{thm-entire}. Finally the remaining subcase is that $d_{L_1}=d_{L_2}=0$, then
$$L_j(s)=\omega Q^{1-2s}\overline{L_j(1-\overline{s})},\quad (j=1,2)$$
where $\omega=1$, and $Q>1$ by Proposition 1. Clearly, there exist positive constants $\sigma_0,C_0$ such that for $s\in \Omega=\{z:\,{\rm Re} z<-\sigma_0, |{\rm Im} z|<-{\rm Re} z\}$, we have
$$|\overline{L_j(1-\overline{s})}-1|\leq C_0 2^{{\rm Re}s},\quad (j=1,2).$$
Then for $s\in \Omega$, we get
$$L_j(s)=\omega Q^{1-2s}(1+O(2^{{\rm Re}s})),\quad (j=1,2)$$
with ${\rm Re}s<0$. Combining this with \eqref{eq-upper} yields that for $\theta\in (\frac{3}{4}\pi,\frac{5}{4}\pi)$, we get
$$
\lim_{r\rightarrow\infty}
\frac{L_j(re^{i\theta})-h(re^{i\theta})}{L_j(re^{i\theta})}=1,\quad (j=1,2)
$$
 so
$$
\lim_{r\rightarrow\infty}\varphi(re^{i\theta})=\lim_{r\rightarrow\infty}\frac{L_1(re^{i\theta})}
{L_2(re^{i\theta})}=1.
$$
Similarly, this property of $\varphi$ gives a contradiction.\vskip 2mm
\par Summing up the above argument, we must have $L_1\equiv L_2$, which completes the proof of Theorem \ref{thm-meromorphic}.\vskip 2mm

\noindent {\bf Proof of Theorem \ref{thm-meromorphic-positive-degree}.}\,\,The case that $\rho(h)>1$ is already done in the proof of Theorem \ref{thm-meromorphic}, so we only deal with the case  that $\rho(h)<1$ in the following. We assume that $L_1\not\equiv L_2$. By Proposition \ref{prop-degree}, it is clear that $d_{L_1}=d_{L_2}=\zeta>0$. Similarly, we have
\begin{equation}\label{small h}
\varphi(s)=\frac{L_1-h}{L_2-h}=(s-1)^{m}e^{as+b}
\end{equation}
for some $a, b\in \mathbb{C}$, $m\in\mathbb{Z}$.\vskip 2mm
\par If $a=0$, then $\varphi$ is rational. Because of Theorem \ref{thm-entire} and Remark \ref{remark-pole}, we focus on the case that $h$ has infinitely many poles $\{p_n\}_{n=1}^{\infty}$ with $p_n\not=1$. Thus, we have $\varphi(p_n)=1$, which implies $\varphi(s)\equiv 1$, a contradiction.\vskip 2mm
\par Now, we have $a\not=0$, and write $a=\rho e^{i\phi}$ for some $\rho>0$ and $\phi\in [0, 2\pi)$.
By Lemma \ref{lemma-estimate}, for any $0<\delta<\frac{\pi}{4}$, we have
\begin{equation}\label{11-04-1}
\begin{split}
\log |L_\eta(re^{i\theta})|=&-(\zeta\cos \theta) r\log r-2\big[\log Q_\eta +\sum_{j=1}^{K_\eta}\lambda_{\eta,j}\log\lambda_{\eta,j}\big]r\cos\theta\\
&+\big[\cos\theta+(\theta-\frac{\pi}{2})\sin\theta\big]\zeta r+O(\log r),\quad (\eta=1,2)
\end{split}
\end{equation}
as $r\to\infty$, uniformly for $\theta\in \mathcal{I}_\delta=(\frac{\pi}{2}+\delta, \pi-\delta)\cup(\pi+\delta, \frac{3}{2}\pi-\delta)$, where $Q_\eta,\lambda_{\eta,j}$ are the constants in the functional equation satisfied by $L_\eta$. Substituting \eqref{eq-upper} and \eqref{11-04-1} into \eqref{small h} yields that
\begin{equation*}
(s-1)^me^{as+b}=\frac{L_1(s)}{L_2(s)}(1+o(1)),\quad \text{as}\,\,r\to\infty,
\end{equation*}
where $s=re^{i\theta}$ with $r\not\in E$ and $\theta\in \mathcal{I}_\delta$. This implies
\begin{equation}\label{eq-nice}
\begin{split}
&\frac{|L_1(re^{i\theta})|}{|L_2(re^{i\theta})|} r^{-m}
|\exp(-\rho re^{i(\theta+\phi)}-b)|\\
&\qquad =r^{-m}|e^{-b}|\exp\left\{[A\cos\theta-\rho\cos(\theta+\phi)]r+O(\log r)\right\}\to 1,
\end{split}
\end{equation}
as $r\not\in E\to\infty$ and $\theta\in \mathcal{I}_\delta$, where the constant $A$ is actually
\[
A=2\left(\log Q_2-\log Q_1\right)+2\left(\sum_{j=1}^{K_2}\lambda_{2, j}\log\lambda_{2, j}
-\sum_{j=1}^{K_1}\lambda_{1, j}\log\lambda_{1, j}\right).\]
It follows from \eqref{eq-nice} that
$$A\cos\theta-\rho\cos(\theta+\phi)=0.$$
Since $\theta$ varies in $\mathcal{I}_\delta$, we must have
$$A-\rho\cos\phi=0,\quad \rho\sin\phi=0.$$
This proves that $\phi=0$ or $\pi$, which means that $a\in \mathbb{R}$.\vskip 2mm
\par
We rewrite the functional equation in the axiom (iii) of $L_1,L_2$ as
\begin{equation}\label{eq-1-1}
L_\eta(s)=\chi_\eta(s)\overline{L_\eta(1-\overline{s})},\quad (\eta=1,2)
\end{equation}
where
$$\chi_\eta(s):=\omega_\eta Q_\eta^{1-2s}\prod_{j=1}^{K_\eta}\frac{\Gamma(\lambda_{\eta,j}
(1-s)+\overline{\mu_{\eta,j}})}{\Gamma(\lambda_{\eta,j}
s+\mu_{\eta,j})},$$
$|\omega_\eta|=1$, $Q_\eta>0, \lambda_{\eta,j}>0$ and $\hbox{Re}\mu_{\eta,j}\geq 0$.
The ``trivial" zeros of $L_\eta(s)$ come from the zeros of $\chi_\eta$, which are the poles of $\Gamma(\lambda_{\eta,j}s+\mu_{\eta,j})(j=1,2,\cdots,K_\eta)$, with the possible exception of $s=0$. Thus, the set $S_\eta$ of all ``trivial" zeros of $L_\eta(s)(\eta=1,2)$ is the union of sets as
\begin{equation}\label{zero-form}S_{\eta,j}=\left\{-\frac{m+\mu_{\eta,j}}{\lambda_{\eta,j}}:\,\, m=0, 1, 2,...\right\},\,\,(j=1,2,\cdots,K_\eta).
\end{equation}
Clearly, the real part of the points in $S_{\eta,j}$ tends to $-\infty$ as $m\to\infty$. Further from \cite{GHK}, there is a constant $D>0$ such that the distance between any distinct two zeros in all $S_{\eta,j}(\eta=1,2;j=1,2,\cdots,K_\eta)$ is more than $D$.
\vskip 2mm
\par We observe that there exists a sequence of common ``trivial" zeros $\{s_n\}_{n=1}^\infty$ of $L_1,L_2$ satisfying $s_n=s_1+(n-1)t$ with some constant $t<0$. Indeed, if there are infinitely many common trivial zeros of $L_1,L_2$, then there exist $j_\eta\in\{1,2,\cdots,K_\eta\}(\eta=1,2)$ such that
$S_{1,j_1}\cap S_{2,j_2}$ contains infinitely many elements. By the property of arithmetic sequence, it is easy to see the existence of $\{s_n\}_{n=1}^\infty$. The remained case is that $L_1,L_2$ have only finitely many common trivial zeros. Then for each $S_{1,j}$, we get the sequence $\{\xi_m\}_{m=N}^\infty$ with $\xi_m=-(m+\mu_{\eta,j})/\lambda_{\eta,j}$ and some positive integer $N$ in which $L_1(\xi_m)=0$ and $L_2(\xi_m)\not=0$.
We can take some $d>0$ such that $L_1$ has no other zeros in the disk $D(\xi_m, d)$
 and $L_2$ has no zeros in $D(\xi_m, d)$. There exists a subsequence $\{m_k\}_{k=1}^{\infty}$ such that $h$ is holomorphic and has no zeros in the closure of $D(\xi_{m_k}, d)$. Otherwise, $$N(r, h)+N(r,1/h)\geq Cr,$$
 where $C$ is a positive constant. It implies that $\rho(h)\geq 1$ which contradicts the condition that $\rho(h)<1$. As in \cite{GHK}, applying Stirling's formula to estimate $\chi_\eta(\eta=1,2)$ in $D(\xi_{m_k}, d)$, together with \eqref{eq-1-1} and $a(1)=1$, it yields that there exists a constant $c>0$ such that if $k$ is sufficiently large,
 \begin{equation}\label{11-07-1}
 |L_\eta(s)|\geq 2^{-|\xi_{m_k}|-d}c d^{K_\eta} Q^{2|\xi_{m_k}|}\exp(|\xi_{m_k}|\sum \lambda_{\eta, j}\log(|\xi_{m_k}\lambda_{\eta, j}|))
 \end{equation}
 for any $s\in \partial D(\xi_{m_k}, d)$ and $\eta=1, 2$. Considering \eqref{eq-upper} or essentially Lemma \ref{meromorphic-estimate}, we see that there exists $d_0>0$ such that if $d<d_0$ and $k$ is large enough, the right hand side of \eqref{11-07-1} is more than $2|h(s)|$. By Rouche's theorem to $L_\eta$ and $L_\eta-h(\eta=1,2)$, we conclude $L_\eta-h$ has the same number of zeros as $L_\eta$ for $\eta=1,2$ in $D(\xi_{m_k}, d)$. Thus, $L_1-h$ has zero in $D(\xi_{m_k}, d)$, while $L_2-h$ has no zero in the same disk. This contradicts our assumption that $L_1-h$ and $L_2-h$ have the same zeros counting multiplicities.\vskip 2mm
\par
 Now, we have a sequence of common trivial zeros $\{s_n\}_{n=1}^{\infty}$ of $L_1$ and $L_2$ with $s_n=s_1+(n-1)t$ for some constant $t<0$. Again, by $\rho(h)< 1$, there exists a subsequence $\{s_{n_k}\}_{k=1}^\infty$ such that $h(s_{n_k})\neq 0, \infty$. It follows that
 $$\varphi(s_{n_k})=1,\quad \text{for\,\,any}\quad m\in \mathbb{N}.$$
 This gives that
 $\arg s=\pi$ is 1-value limiting direction of $\varphi(s)=(s-1)^{m}e^{as+b}$ with $a\in \mathbb{R}\backslash\{0\}$. Since $(s-1)^m e^{as+b}(a\in \mathbb{R}\backslash\{0\})$ has only two possible 1-value limiting direction $\theta=\frac{\pi}{2}$ or $\theta=\frac{3\pi}{2}$. It is a contradiction.\vskip 2mm
\par Therefore, summing up the above argument gives $L_1\equiv L_2$, which completes the proof of Theorem \ref{thm-meromorphic-positive-degree}.

\end{document}